\documentclass[a4paper]{amsart}
\usepackage{amsthm,amsfonts,amsmath,amssymb}
\usepackage[abs]{overpic}
\usepackage{thm-restate}

\newtheorem{theorem}{Theorem}[section]

\newtheorem{proposition}[theorem]{Proposition}

\newtheorem{claim}[theorem]{Claim}

\newtheorem{example}[theorem]{Example}

\newtheorem{remark}[theorem]{Remark}

\begin{document}
\title[DELTA-UNKNOTTING NUMBER FOR PRETZEL KNOTS]{DELTA-UNKNOTTING NUMBER FOR PRETZEL KNOTS}
\author{Kazumichi Nakamura}
\email{mi-ka-07130703@ezweb.ne.jp}





\begin{abstract}
The $\Delta$-unknotting number for a knot is defined as the minimum number of $\Delta$-moves needed to deform the knot into the trivial knot. It is known that, for positive pretzel knots, the $\Delta$-unknotting number coincides with the second coefficient of their Conway polynomial. In this paper, we compute the $\Delta$-unknotting number for positive pretzel knots.
As a consequence of the above result, among positive pretzel knots of odd type with a fixed crossing number $n$, where $n$ is odd, the ${\Delta}$-unknotting number is maximized by $P(1, 1, ... , 1) \quad \big( \cong T(2,n) \big)$, and the maximum value is $\frac{1}{8}(n^2 - 1)$.
We also obtain a similar result for torus knots.
We further determine the $\Delta$-unknotting number for pretzel knots of type $P(-1, p_2, ..., p_n)$, where $p_i$ is a positive odd integer for $2 \leq i \leq n$ and $n$ is odd.
\end{abstract}

\maketitle

\noindent\textbf{keywords:}$\Delta$-move, $\Delta$-unknotting number, Conway polynomial, pretzel knot, positive pretzel knot, torus knot, linking number.

\section{Introduction}
In this paper, we study the $\Delta$-unknotting numbers for pretzel knots.

In \cite{MaN}, H. Murakami and Y. Nakanishi introduced a local move on regular diagrams of oriented knots and links, called a $\Delta$-move (or $\Delta$-unknotting operation), as illustrated in Figure \ref{fig:delta}.
\begin{figure}[htbp]
    \centering
    \includegraphics[width=0.8\linewidth]{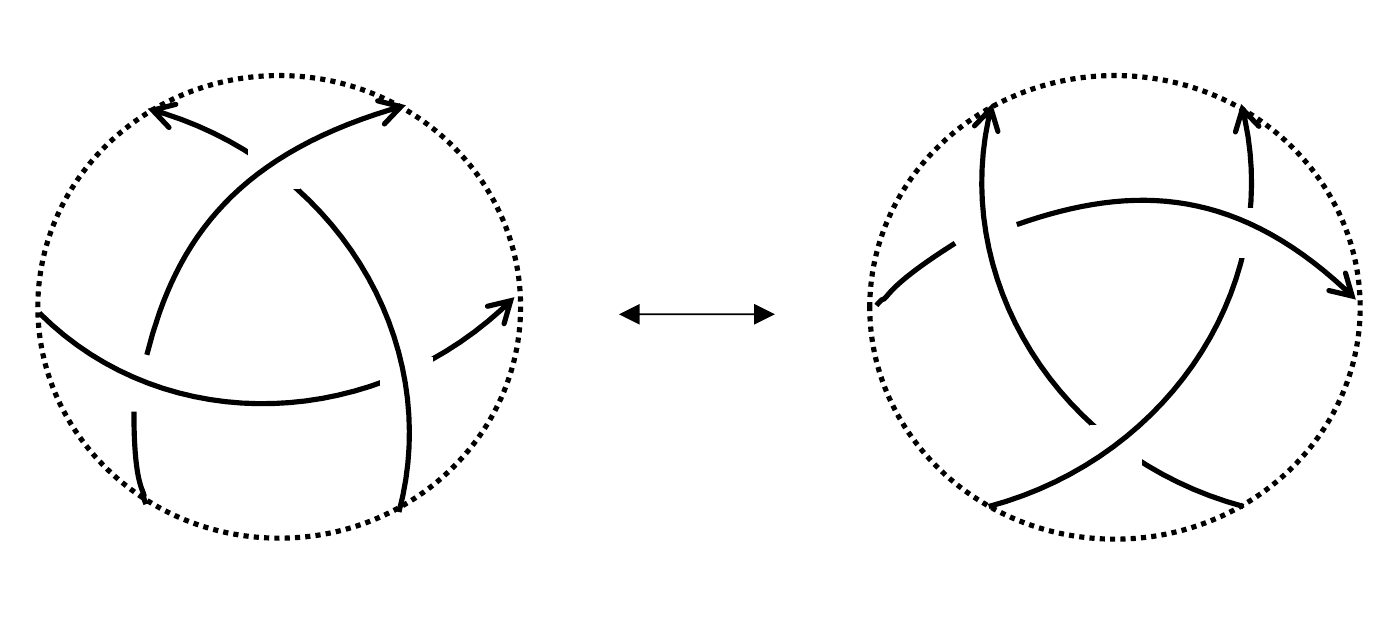}
    \caption{}
    \label{fig:delta}
\end{figure}
They proved that two knots can be deformed into each other by a finite sequence of $\Delta$-moves. 
The $\Delta$-Gordian distance $d_G^{\Delta}(K, K^{'})$ of two oriented knots $K$ and $K^{'}$ is defined as the minimum number of $\Delta$-moves needed to deform a diagram of $K$ into that of $K{'}$. 
The $\Delta$-unknotting number $u^{\Delta}(K)$  of an oriented knot $K$ is defined as the $\Delta$-Gordian distance $d_G^{\Delta}(K, O)$ of $K$ and the trivial knot $O$.
In \cite{Oka1}, M. Okada proved that $u^{\Delta}(K) \geq |a_2(K)|$, where $a_2(K)$ denotes the second coefficient of the Conway polynomial of $K$. 

In \cite{NaNaU}, joint work with Y. Nakanishi and Y. Uchida,
we determined the $\Delta$-unknotting numbers for certain classes of oriented knots: torus knots, positive pretzel knots, and positive $3$-braids.
For these classes, the $\Delta$-unknotting number coincides with the second coefficient of the Conway polynomial. In particular, the $\Delta$-unknotting number for $T(p,q)$ (the torus knot of type $(p,q)$) is equal to $(p^2-1)(q^2-1)/24$.

Subsequently, the $\Delta$-unknotting number for a subfamily of Turk's head knots was computed in \cite{NaYa}.
Recently, in \cite{Naka3}, we computed the $\Delta$-unknotting numbers for two-bridge knots of type
$C(2\beta_1, 2\beta_2, ... ,  2\beta_n)$ and type $C(2\beta_1, 2\beta_2, ... , 2\beta_{n-1}, 2\beta_n-1)$, where $\beta_i$ is a positive integer for $1 \leq i \leq n$.

\vspace{1em}
In this paper, we compute the $\Delta$-unknotting number for positive pretzel knots.
We consider two cases, the odd type and the even type.
Theorem~\ref{thm:odd} and Theorem~\ref{thm:even} deal with positive pretzel knots of odd and even types, respectively.
\begin{restatable}{theorem}{odd} \label{thm:odd}
Let $P=P(p_1, p_2, ... , p_n)$ be a positive pretzel knot of odd type, where $p_i$ is a positive odd integer for $1 \leq i \leq n$ and $n$ is odd.
Then, we have
\begin{equation*}
u^{\Delta}(P) = 
\frac{1}{4} \sum_{1 \leq i<j }^{n} p_{i}p_{j} +  \frac{1}{8} (n-1) \quad \big(= a_2(P)\big).
\end{equation*}
\end{restatable}

\begin{restatable}{theorem}{even} \label{thm:even}
Let $P=P(p_1, p_2, ... , p_n)$ be a positive pretzel knot of even type, where $p_1$ is an even integer, and $p_i$ is a positive odd integer for $2 \leq i \leq n$. 
Then, we have
\begin{enumerate}
    \item
    If $n$ is even and $p_1 > 0$, then
    \[
    u^{\Delta}(P)=
\frac{1}{8} \sum_{i=1}^{n} p_{i}^2 + \frac{1}{4} p_1 \sum_{i=2}^{n} p_i - \frac{1}{8} (n-1) \quad \big(= a_2(P)\big).
\]
    \item
    If $n$ is odd and $p_1 < 0$, then
    \[
    u^{\Delta}(P)=
\frac{1}{8} \sum_{i=2}^{n} p_{i}^2 - \frac{1}{4} p_1 \sum_{i=2}^{n} p_i - \frac{1}{8} (n-1) \quad \big(= a_2(P)\big).
\]
\end{enumerate}
\end{restatable}

To prove Theorems \ref{thm:odd} and \ref{thm:even}, we first present Lemmas \ref{lemma:odd} and \ref{lemma:even}.

Lemmas~\ref{lemma:odd} and~\ref{lemma:even} apply to pretzel knots $P$ of odd and even types, respectively; in either case, 
$P$ is not necessarily a positive pretzel knot.

\begin{restatable}{lemma}{oddl} \label{lemma:odd}
Let $P=P(p_1, p_2, ... , p_n)$ be a pretzel knot of odd type, where $p_i$ is an odd integer for $1 \leq i \leq n$ and $n$ is odd.
Then, we have 
\begin{equation*}
a_2(P) = 
\frac{1}{4} \sum_{1 \leq i<j }^{n} p_{i}p_{j} +  \frac{1}{8} (n-1)
\quad \big(= 
\frac{1}{8} \{ (\sum_{i=1}^{n} p_{i})^2 -\sum_{i=1}^{n} p_{i}^2 \} + \frac{1}{8} (n-1) \big).
\end{equation*}
\end{restatable}

\begin{restatable}{lemma}{evenl} \label{lemma:even}
Let $P=P(p_1, p_2, ... , p_n)$ be a pretzel knot of even type, where $p_1$ is an even integer, and $p_i$ is an odd integer for $2 \leq i \leq n$. 
Then, we have
\begin{enumerate}
    \item
    If $n$ is even, then
    \[
    a_2(P) =
\frac{1}{8} \sum_{i=1}^{n} p_{i}^2 + \frac{1}{4} p_1 \sum_{i=2}^{n} p_i - \frac{1}{8} (n-1).
\]
    \item
    If $n$ is odd, then
    \[
    a_2(P) =
\frac{1}{8} \sum_{i=2}^{n} p_{i}^2 - \frac{1}{4} p_1 \sum_{i=2}^{n} p_i - \frac{1}{8} (n-1).
\]
\end{enumerate}
\end{restatable}

By the way, there exist pretzel knots 
$P=P(p_1, p_2, ... , p_n)$ that are not positive pretzel knots, yet satisfy $u^{\Delta}(P) = |a_2(P)|$.
In particular, if $P$ is of odd type with $p_1 = -1$ and $p_i$ is a positive odd integer for $2 \leq i \leq n$, then $u^{\Delta}(P) = a_2(P)$, as stated in Theorem~\ref{thm:oddone}.
This example suggests that the equality between the $\Delta$-unknotting number and the second coefficient of the Conway polynomial holds beyond the class of positive pretzel knots.

\begin{restatable}{theorem}{oddone} \label{thm:oddone}
Let $P=P(-1, p_2, ... , p_n)$ be a pretzel knot of odd type, where $p_i$ is a positive odd integer for $2 \leq i \leq n$ and $n$ is odd.
Then, we have 
\begin{equation*}
u^{\Delta}(P) = - \frac{1}{4} \sum_{i = 2}^{n} p_{i} + 
\frac{1}{4} \sum_{2 \leq i<j }^{n} p_{i}p_{j} +  \frac{1}{8} (n-1) \quad \big(= a_2(P)\big).
\end{equation*}
\end{restatable}

\vspace{1em}
By applying Theorems \ref{thm:odd}, \ref{thm:even}, and the formula 
$u^{\Delta}(T(p,q)) = \frac{1}{24} (p^2-1)(q^2-1)$, we obtain the following theorems (Theorems \ref{thm:deltamax} and \ref{thm:deltaone}).

\begin{restatable}{theorem}{deltamax}\label{thm:deltamax}
The following statements hold.

\begin{enumerate}

\item
Among positive pretzel knots with a fixed crossing number $n$,
the $\Delta$-unknotting number is maximized as follows.

\begin{enumerate}
\renewcommand{\labelenumii}{\textnormal{(\roman{enumii})}}

\item
Suppose that the knot is of odd type and that $n$ is odd.
Then the maximum is attained by
$P(\underbrace{1,1,\dots,1}_{n})$,
and its value is
$\frac{1}{8}(n^{2}-1)$.

\item
Suppose that the knot is of even type~A, where $n$ is odd and
$p_{1}$ is a positive even integer satisfying $p_{1}<n$.
Then the maximum is attained by
$P(p_{1},\, n-p_{1})$,
and its value is
$\frac{1}{8}(n^{2}-1)$.

\item
Suppose that the knot is of even type~B, where $n\ge 8$ is even.
Then the maximum is attained by
$P(-2, 3, n-5)$,
and its value is
$\frac{1}{8}(n^{2}-6n+24)$.

\end{enumerate}

\item
Among torus knots with a fixed odd crossing number $n$,
the $\Delta$-unknotting number is maximized by
$T(2,n)$,
and the maximal value is
$\frac{1}{8}(n^{2}-1)$.

\end{enumerate}
\end{restatable}

Here, $P(\underbrace{1, 1, ... , 1}_{\text{$n$}}) \cong P(p_1, n-p_1) \cong T(2,n)$.

\begin{restatable}{theorem}{deltamin} \label{thm:deltamin}
Among positive pretzel knots of odd type with a fixed odd crossing number $n$, the ${\Delta}$-unknotting number is minimized by $P(1, 1, n-2)$, and the minimum value is $\frac{1}{2}(n-1)$.
\end{restatable}

\begin{restatable}{theorem}{deltaone} \label{thm:deltaone}
The following statements hold.
\begin{enumerate}
\item A positive pretzel knot has $\Delta$-unknotting number one if and only if it is $3_1$ or its mirror image.
\item A torus knot has $\Delta$-unknotting number one if and only if it is $3_1$ or its mirror image.
\end{enumerate}
\end{restatable}

From Theorem \ref{thm:deltamax}, we make the following conjectures (Conjectures \ref{conj:conjcrtn} and \ref{conj:conjcrposi}).

\begin{restatable}{conjecture}{conjcrtn} \label{conj:conjcrtn}
Among knots with a fixed crossing number $n$, where $n$ is odd, the ${\Delta}$-unknotting number is maximized by $T(2,n)$, and the maximum value is $\frac{1}{8}(n^2 - 1)$.
\end{restatable}

\begin{restatable}{conjecture}{conjcrposi} \label{conj:conjcrposi}
Among knots with a fixed crossing number different from $4$ and $6$, the $\Delta$-unknotting number is maximized by positive knots.
\end{restatable}

For prime knots with at most 10 crossings, Conjectures 
\ref{conj:conjcrtn} and
\ref{conj:conjcrposi} 
hold.

In \cite{Tani}, several relations between knot invariants are established, one of which is related to Conjectures
\ref{conj:conjcrtn} and \ref{conj:conjcrposi}.
In particular, for a knot $K$ with crossing number $c(K) \geq 4$, one has
$u^{\Delta}(K) \leq \frac{1}{4} \big(c(K)^2 - 2c(K) - 3 \big)$.

Concerning the ordinary unknotting operation,
it is well known that every nontrivial knot $K$ satisfies the inequality $u(K) \leq \frac{1}{2} \big(c(K) - 1 \big)$.
Moreover, equality holds if and only if $K$ is $T(2,n)$ for some odd integer $n$ (see \cite{Tani2}).

From Theorem \ref{thm:deltaone}, we make the following conjecture (Conjecture \ref{conj:conjpkdone}).

\begin{restatable}{conjecture}{conjpkdone} \label{conj:conjpkdone}
A positive knot has $\Delta$-unknotting number one if and only if it is $3_1$.
\end{restatable}

For prime knots with at most 10 crossings, Conjecture \ref{conj:conjpkdone} holds.

In contrast, from the result of \cite{Ta}, we obtain the following corollary for almost positive knots.
\begin{restatable}{corollary}{apkone} \label{thm:apkone}
There is no almost positive knot with ${\Delta}$-unknotting number one. 
\end{restatable}

\vspace{1em}
As mentioned above, for torus knots, positive pretzel knots, and positive $3$-braids, the $\Delta$-unknotting number coincides with the second coefficient of the Conway polynomial.
Therefore, we make the following conjectures 
(Conjectures \ref{conj:conjpd} and \ref{conj:conjdd}).

\begin{restatable}{conjecture}{conjpd} \label{conj:conjpd}
If a prime knot $K$ is positive, then $u^{\Delta}(K) = |a_2(K)|$. 
\end{restatable}

\begin{restatable}{conjecture}{conjdd} \label{conj:conjdd}
If a minimal crossing diagram of a knot $K$ is positive, then $u^{\Delta}(K) = |a_2(K)|$. 
\end{restatable}

Conjecture \ref{conj:conjpd} concerns the positivity of knots, whereas 
Conjecture \ref{conj:conjdd} concerns the positivity of minimal crossing diagrams.
For prime knots with at most 10 crossings, Conjectures
\ref{conj:conjpd} and \ref{conj:conjdd}
hold.

\vspace{1em}
In addition, in Section \ref{sec5}, we give examples demonstrating the usefulness of the inequality $u^{\Delta}(K) \geq |\sigma(K)|/2$ from \cite{MaN}, where $\sigma(K)$ denotes the signature of a knot $K$.

\section{Preliminaries}\label{sec2}

A pretzel link $P(p_1, p_2, ... , p_n)$ is a link obtained by connecting $n$ twisted bands, where each band has $p_i$ half-twists.
The sign of $p_i$ determines the direction of the $i$th twist region
(see Figure \ref{fig:pretzel}).
A pretzel link is a knot if and only if precisely one $p_i$ is even, or all $p_i$ are odd and $n$ is odd.

A positive pretzel knot is a pretzel knot whose standard diagram has either all positive crossings or all negative crossings, as follows.
There are two types of positive pretzel knots: the odd type and the even type.
For a positive pretzel knot of odd type, both $n$ and all $p_i$ are odd and positive.
For a positive pretzel knot of even type, one of the following holds:
\begin{itemize}
    \item even type $A$ : $n$ is even and positive, $p_1$ is even and positive, and $p_i$ $(2 \leq i \leq n)$ are odd and positive; or
    \item even type $B$ : $n$ is odd and positive, $p_1$ is even and negative, and $p_i$ $(2 \leq i \leq n)$ are odd and positive.
\end{itemize}
\begin{figure}
    \centering
    \includegraphics[width=1\linewidth]{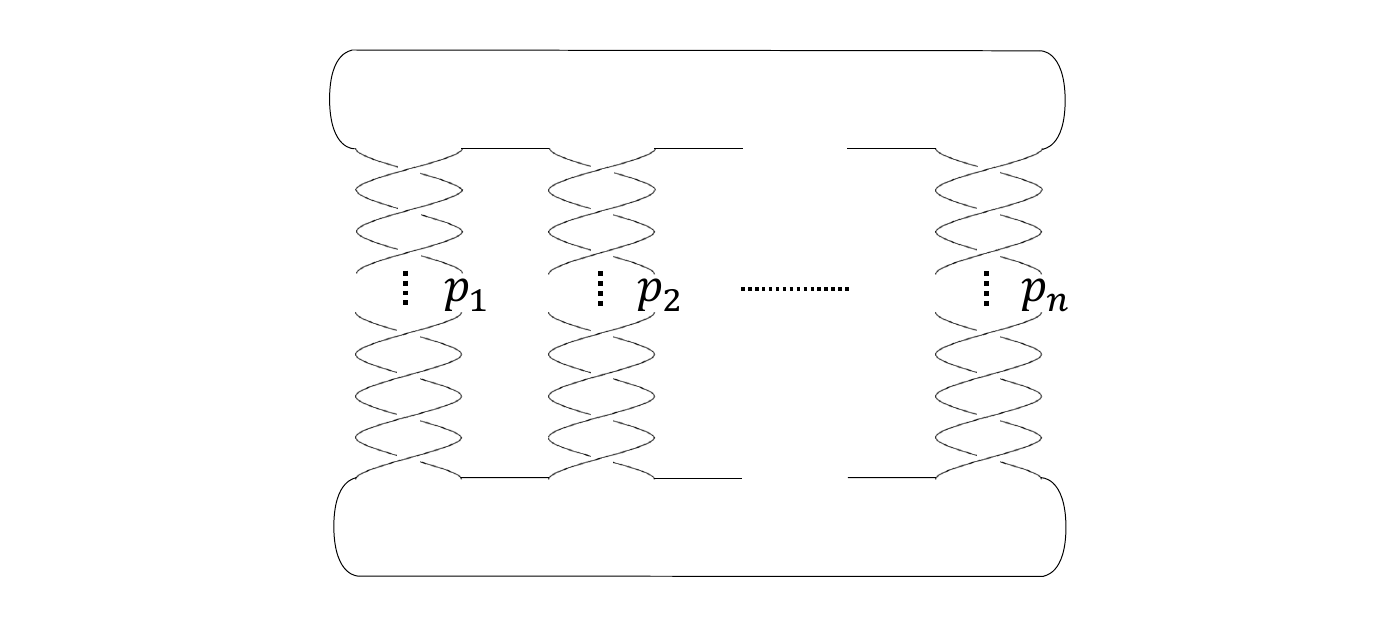}
    \caption{}
    \label{fig:pretzel}
\end{figure}

In this paper, we refer to \cite{DaM, Kaw} for the notation of pretzel knots.

We use the technique in \cite{MaN, NaNaU}.

\begin{claim}[\cite{MaN}]\label{claim:clasp}
A clasp can leap over a hurdle by a single $\Delta$-move.
\end{claim}

The meaning is given in Figure \ref{fig:clasp}. Here, the clasp is vertically standing on the 2-sphere on which the other part of the knot diagram lies (except for crossings). We use Claim \ref{claim:clasp} to show Claim \ref{claim:change}.
\begin{figure}[htbp]
 \centering
 \includegraphics[width=0.6\linewidth]{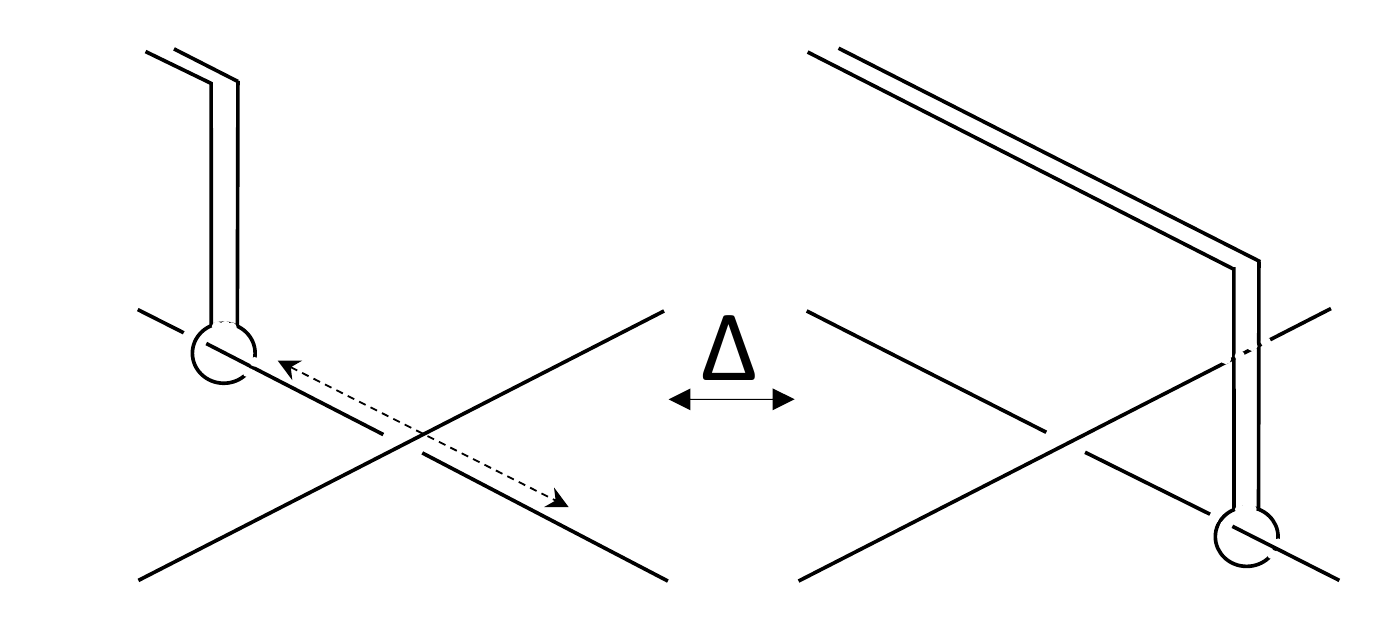}
 \caption{}
 \label{fig:clasp}
\end{figure}

\begin{claim}[\cite{NaNaU}]\label{claim:change}
We can exchange a crossing of a knot by a finite number of $\Delta$-moves.
Moreover, the number of $\Delta$-moves needed is equal to the number of hurdles.
\end{claim}

The meaning and the proof are given in Figure \ref{fig:tdelta}. Here, the clasp leaps over $t$ hurdles by $t$ times $\Delta$-moves. Then, the crossing labeled with the asterisk * is deformed into the opposite crossing.
\begin{figure}[htbp]
 \centering
 \includegraphics[width=1\linewidth]{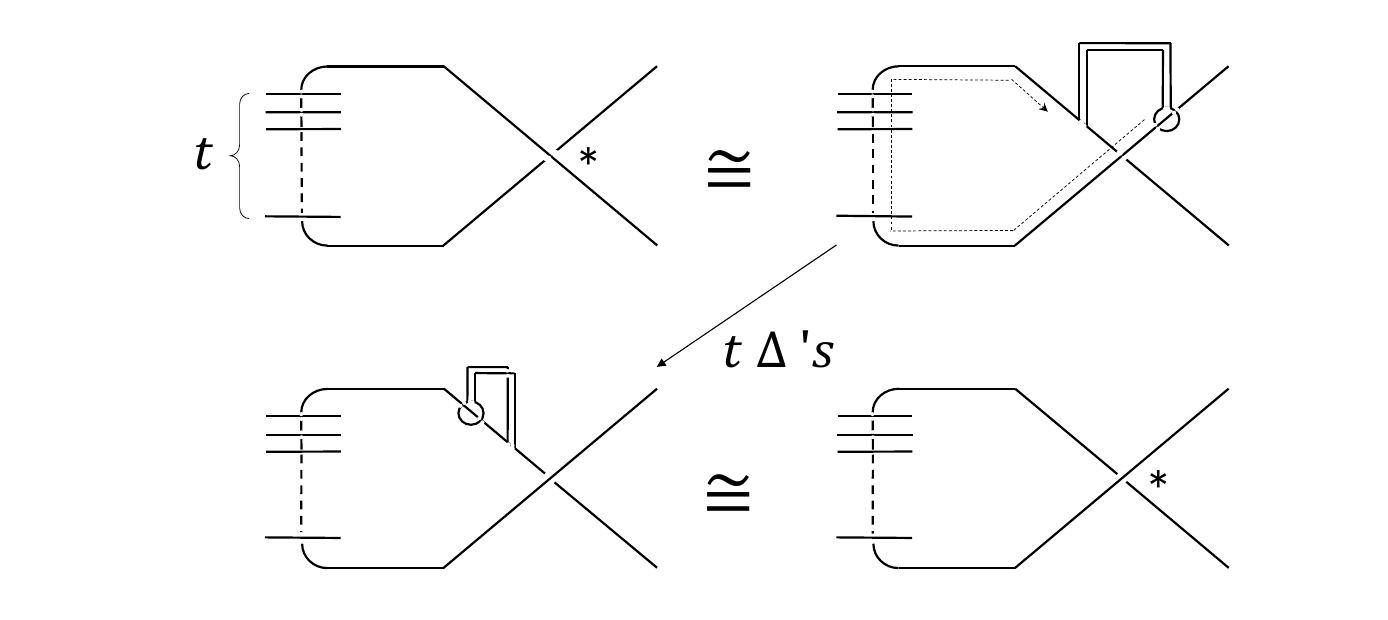}
 \caption{}
 \label{fig:tdelta}
\end{figure}

In particular, Claim \ref{claim:technique} holds for a class of pretzel knots.
\begin{claim}\label{claim:technique}
Let $P(-1, p_2, ... , p_n)$ be a pretzel knot of odd type, where $p_i$ is a positive odd integer for $2 \leq i \leq n$ and $n$ is odd.
In Figure \ref{fig:ccpretzel}, if we perform $\frac{1}{2} ( \sum_{i=3}^{n}p_i - 1 )$ times $\Delta$-moves, we can exchange the crossing labeled with the asterisk *.
Then, we have 
$d_G^{\Delta}(P(-1, \underline{p_2}, p_3, ... , p_n), P(-1, \underline{p_2-2}, p_3, ... , p_n)) \leq \frac{1}{2} ( \sum_{i=3}^{n}p_i - 1 )$.
\end{claim}
\begin{figure}
    \centering
    \includegraphics[width=1\linewidth]{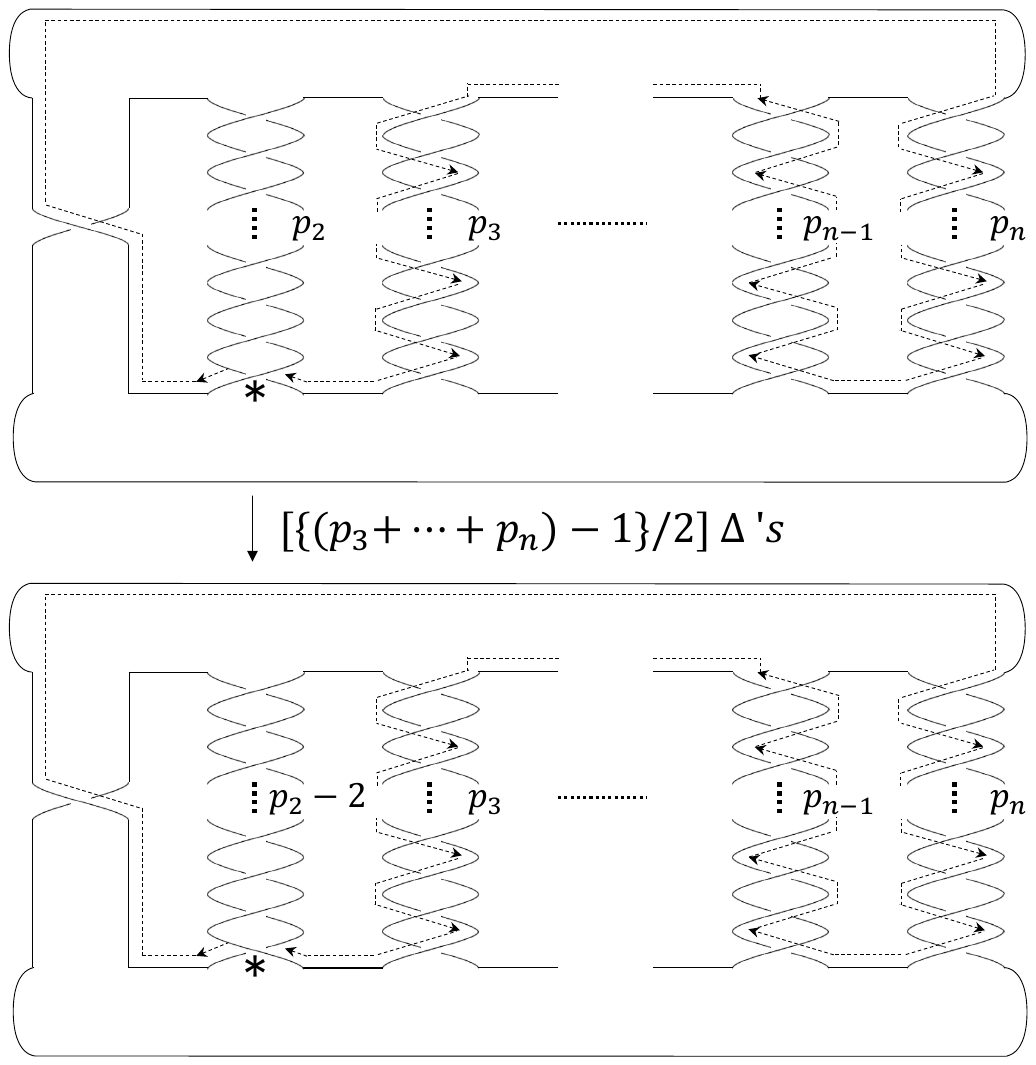}
    \caption{}
    \label{fig:ccpretzel}
\end{figure}

In this paper, we use Propositions \ref{prop:lk}, \ref{prop:a2}, \ref{prop:torus}, and \ref{prop:pretzel}.

\begin{proposition}\label{prop:lk}
Let $k_+$ be a knot, $k_-$ the knot obtained from $k_+$ by exchanging a positive crossing into a negative crossing, and $k_0$ a $2$-component link obtained from $k_+$ by smoothing at the crossing. Then $a_2(k_+)-a_2(k_-) = lk(k_0)$.
\end{proposition}

The proof can be found in [2, Chap. III].

\begin{proposition}[\cite{Oka1}]\label{prop:a2}
For any two knots $K$ and $K^{'}$, the difference $d_G^{\Delta}(K,K^{'}) - |a_2(K) - a_2(K^{'})|$ is a non-negative even integer. 
In particular, the difference $u^{\Delta}(K) -|a_2(K)|$ is also a non-negative even integer.
\end{proposition}

\begin{proposition}[\cite{NaNaU}]\label{prop:torus}
Let $T(p,q)$ be the torus knot of type $(p,q)$ for a pair of positive integers $p,q$ with $\gcd(p,q)=1$.
Then we have 
\[
u^{\Delta}(T(p,q)) = \frac{1}{24} (p^2-1)(q^2-1) \quad \big( = a_2(T(p,q)) \big).
\]
\end{proposition}

\begin{proposition}[\cite{NaNaU}]\label{prop:pretzel}
Let $P=P(p_1, p_2, ... , p_n)$ be a positive pretzel knot.
Then, we have $u^{\Delta}(P) = a_2(P)$.
\end{proposition}

Since $a_2(K)=a_2(K^*)$ and $u^{\Delta}(K)=u^{\Delta}(K^*)$, where $K^*$ denotes the mirror image of a knot $K$, we do not distinguish a knot from its mirror image in this paper.

\section{Proofs of Theorems \ref{thm:odd} and \ref{thm:even}}\label{sec3}

In this section, we give proofs of Theorems~\ref{thm:odd} and~\ref{thm:even}.
To this end, we first establish Lemmas~\ref{lemma:odd} and~\ref{lemma:even} by applying Proposition~\ref{prop:lk}.
We begin with the following lemma.

\subsection{Proof of Lemma \ref{lemma:odd}}
\label{subsec31}

\oddl*
\begin{proof}
If $n=1$, then $P=P(p_1)$ is the trivial knot $O$, and hence $a_2(P)=0$. In the following, we assume $n \geq 3$.

First, we consider the case where $p_i > 0$ for $1 \leq i \leq n$.
Let $k_+ = P(\underline{p_1}, p_2, ... , p_n)$ and $k_- = P(\underline{p_1-2}, p_2, ... , p_n)$.

By Proposition \ref{prop:lk}, we have $a_2(P(\underline{p_1}, p_2, ... , p_n)) - a_2(P(\underline{p_1-2}, p_2, ... , p_n)) = \frac{1}{2} \sum_{i=2}^{n} p_i$.

By repeating the same computation $\frac{1}{2}(p_1 - 1)$ times, we have
\begin{align*}
a_2(P(\underline{p_1}, p_2, ... , p_n))
& - a_2(P(\underline{p_1-2}, p_2, ..., p_n)) = \frac{1}{2}  \sum_{i=2}^{n} p_i \quad, ... , \\
a_2(P(\underline{3}, p_2, ... , p_n)) 
& - a_2(P(\underline{1}, p_2, ... , p_n)) = \frac{1}{2}  \sum_{i=2}^{n} p_i.
\end{align*}

By continuing with another $\frac{1}{2} (p_2 - 1)$ steps, we have
\begin{align*}
a_2(P(1, \underline{p_2}, p_3, ... , p_n)) 
& - a_2(P(1, \underline{p_2-2}, p_3, ... , p_n)) 
= \frac{1}{2} (1 +\sum_{i=3}^{n} p_i) \quad, ... , \\
a_2(P(1, \underline{3}, p_3, ... , p_n)) 
& - a_2(P(1, \underline{1}, p_3, ... , p_n)) 
= \frac{1}{2} (1 + \sum_{i=3}^{n} p_i).
\end{align*}

Proceeding further, we have 
\begin{align*}
a_2(P(\overbrace{1, 1, ..., 1}^{\text{$n-1$}}, \underline{p_n}))
& - a_2(P(\overbrace{1, 1, ... , 1}^{\text{$n-1$}}, \underline{p_n-2}))
= \frac{1}{2} ( n - 1 ) 
\quad, ... ,  \\
a_2(P(\overbrace{1, 1, ... , 1}^{\text{$n-1$}}, \underline{3})) 
& - a_2(P(\overbrace{1, 1, ... , 1}^{\text{$n-1$}}, \underline{1})) 
= \frac{1}{2} ( n - 1 ).
\end{align*}

Here, $P(\overbrace{1, 1, ... , 1}^{\text{$n$}}) \cong T(2, n)$, hence $a_2(P(\overbrace{1, 1, ... , 1}^{\text{$n$}}) ) = \frac{1}{8} ( n^2 - 1)$.

By summing up, we obtain 
\begin{align*}
a_2(P(p_1, p_2, ... , p_n)) 
& = \frac{1}{4} \sum_{1=i<j}^{n} p_{i}p_{j} - \frac{1}{8} (n-1)n + \frac{1}{8} ( n^2 - 1)\\
& = \frac{1}{4} \sum_{1=i<j}^{n} p_{i}p_{j} + \frac{1}{8} (n-1).
\end{align*}

In particular, in the case where $p_i > 0$, we set $k_+ = P(\overbrace{1, ... , 1}^{\text{$i-1$}}, \underline{p_i}, p_{i+1}, ... , p_n)$ and $k_- = P(\overbrace{1, ... , 1}^{\text{$i-1$}}, \underline{p_i - 2}, p_{i+1}, ..., p_n)$.

By Proposition \ref{prop:lk}, we have $a_2(P(\overbrace{1, ... , 1}^{\text{$i-1$}}, \underline{p_i}, p_{i+1}, ... , p_n))
- a_2(P(\overbrace{1, ... , 1}^{\text{$i-1$}}, \underline{p_i - 2}, p_{i+1}, ... , p_n)) 
= \frac{1}{2} \{ ( i - 1 ) + \sum_{k=i+1}^{n} p_k$ \}.

By repeating the same computation 
$\frac{1}{2}(p_i - 1)$ times, we have
\begin{align*}
& a_2(P(\overbrace{1, 1, ..., 1}^{\text{$i-1$}}, \underline{p_i}, p_{i+1}, ... , p_n)) 
 - a_2(P(\overbrace{1, 1, ... , 1}^{\text{$i-1$}}, \underline{p_i - 2}, p_{i+1},  ... , p_n)) \\
& = \frac{1}{2} \{ ( i - 1 ) + \sum_{k=i+1}^{n} p_k \}
\quad, ... ,  \\
& a_2(P(\overbrace{1, 1, ... , 1}^{\text{$i-1$}}, \underline{3}, p_{i+1}, ... , p_n))  
 - a_2(P(\overbrace{1, 1, ... , 1} ^{\text{$i-1$}} , \underline{1}, p_{i+1}, ... ,p_n))  \\
& =  \frac{1}{2} \{ ( i - 1 ) + \sum_{k=i+1}^{n} p_k \}.
\end{align*}

By summing up, we obtain 
\begin{align*}
& a_2(P(\overbrace{1, 1, ..., 1}^{\text{$i-1$}}, \underline{p_i}, p_{i+1}, ... , p_n)) 
 - a_2(P(\overbrace{1, 1, ... , 1}^{\text{$i-1$}}, \underline{1}, p_{i+1}, ... , p_n)) \\
& = \frac{1}{4} \{ ( i - 1 ) + \sum_{k=i+1}^{n} p_k \}(p_i - 1).
\end{align*}

By the way, in the case where $p_i < 0$, we set $k_+ = P(\overbrace{1, ... , 1}^{\text{$i-1$}}, \underline{p_i + 2}, p_{i+1}, ... , p_n)$ and $k_- = P(\overbrace{1, ... , 1}^{\text{$i-1$}}, \underline{p_i}, p_{i+1}, ..., p_n)$.

By Proposition \ref{prop:lk}, we have $a_2(P(\overbrace{1, ... , 1}^{\text{$i-1$}}, \underline{p_i + 2}, p_{i+1}, ... , p_n))
- a_2(P(\overbrace{1, ... , 1}^{\text{$i-1$}}, \underline{p_i}, p_{i+1}, ... , p_n))
= \frac{1}{2} \{ ( i - 1 ) + \sum_{k=i+1}^{n} p_k$ \}.

By repeating the same computation 
$- \frac{1}{2}(p_i + 1) + 1$ times, we have
\begin{align*}
& a_2(P(\overbrace{1, 1, ..., 1}^{\text{$i-1$}}, \underline{p_i}, p_{i+1}, ... , p_n)) 
 - a_2(P(\overbrace{1, 1, ... , 1}^{\text{$i-1$}}, \underline{p_i + 2}, p_{i+1}, ... , p_n)) \\
& = - \frac{1}{2} \{ ( i - 1 ) + \sum_{k=i+1}^{n} p_k \}
\quad, ... ,  \\
& a_2(P(\overbrace{1, 1, ... , 1}^{\text{$i-1$}}, \underline{-3}, p_{i+1}, ... , p_n))  
 - a_2(P(\overbrace{1, 1, ... , 1} ^{\text{$i-1$}} , \underline{-1}, p_{i+1}, ... ,p_n))  \\
& =  - \frac{1}{2} \{ ( i - 1 ) + \sum_{k=i+1}^{n} p_k \} \\
& a_2(P(\overbrace{1, 1, ... , 1}^{\text{$i-1$}}, \underline{-1}, p_{i+1}, ... , p_n))  
 - a_2(P(\overbrace{1, 1, ... , 1} ^{\text{$i-1$}} , \underline{1}, p_{i+1}, ... ,p_n))  \\
& =  - \frac{1}{2} \{ ( i - 1 ) + \sum_{k=i+1}^{n} p_k \}.
\end{align*}

By summing up, we obtain 
\begin{align*}
& a_2(P(\overbrace{1, 1, ..., 1}^{\text{$i-1$}}, \underline{p_i}, p_{i+1}, ... , p_n)) 
 - a_2(P(\overbrace{1, 1, ... , 1}^{\text{$i-1$}}, \underline{1}, p_{i+1}, ... , p_n)) \\
& = \frac{1}{4} \{ ( i - 1 ) + \sum_{k=i+1}^{n} p_k \}(p_i - 1).
\end{align*}

Therefore, regardless of whether $p_i > 0$ or $p_i <0$, the cumulative effect remains the same.

The proof is complete.
\end{proof}

\subsection{Proof of Lemma \ref{lemma:even}}\label{subsec32}

\evenl*

\begin{proof}
Proceeding in the same way as in the proof of Lemma \ref{lemma:odd}, we obtain Lemma \ref{lemma:even}.

(1) First, we consider the case where $p_1 > 0$.
Let $k_+ = P(\underline{p_1}, p_2, ... , p_n)$ and $k_- = P(\underline{p_1 - 2}, p_2, ... , p_n)$.

By Proposition \ref{prop:lk}, we have $a_2(P(\underline{p_1}, p_2, ... , p_n))
- a_2(P(\underline{p_1-2}, p_2, ..., p_n)) = \frac{1}{2} \{ (p_1 -1) + \sum_{i=2}^{n} p_i$ \}.

By repeating the same computation  $\frac{1}{2} p_1$ times, we have
\begin{align*}
a_2(P(\underline{p_1}, p_2, ... , p_n))
& - a_2(P(\underline{p_1-2}, p_2, ..., p_n)) = \frac{1}{2} \{ (p_1 -1) + \sum_{i=2}^{n} p_i \}\quad, ... , \\
a_2(P(\underline{2}, p_2, ... , p_n)) 
& - a_2(P(\underline{0}, p_2, ... , p_n)) = \frac{1}{2} ( 1 + \sum_{i=2}^{n} p_i ).
\end{align*}

Here, $P(0, p_2, ... , p_n) \cong
T(2, p_2) \# ... \# T(2, p_n)$, 
hence $a_2(P(0, p_2, ... ,p_n)) = a_2(T(2, p_2)) + ... + a_2(T(2, p_n)) = \frac{1}{8} \sum_{i=2}^{n} (p_i^2 - 1)$.

By summing up, we obtain 
\begin{align*}
a_2(P(p_1, p_2, ... , p_n)) 
& = \frac{1}{4} p_1 \sum_{i=2}^{n} p_{i} + \frac{1}{8} p_1^2 + \frac{1}{8} \sum_{i=2}^{n} (p_i^2 -1)\\
& = \frac{1}{8} \sum_{i=1}^{n} p_i^2 + \frac{1}{4} p_1 \sum_{i=2}^{n} p_{i} - \frac{1}{8} (n-1). 
\end{align*}

In the case where $p_1 < 0$, we set $k_+ = P(\underline{p_1 + 2}, p_2, ... , p_n)$ and $k_- = P(\underline{p_1}, p_2 , ... , p_n)$.

By Proposition \ref{prop:lk}, we have $a_2(P(\underline{p_1 + 2}, p_2, ... , p_n))
- a_2(P(\underline{p_1}, p_2, ..., p_n)) = \frac{1}{2} \{ (p_1 + 1) + \sum_{i=2}^{n} p_i$ \}.

By repeating the same computation 
$- \frac{1}{2} p_1$ times, we have
\begin{align*}
a_2(P(\underline{p_1}, p_2, ... , p_n))
& - a_2(P(\underline{p_1 + 2}, p_2, ..., p_n)) = - \frac{1}{2} \{ (p_1 + 1) + \sum_{i=2}^{n} p_i \}\quad, ... , \\
a_2(P(\underline{-2}, p_2, ... , p_n)) 
& - a_2(P(\underline{0}, p_2, ... , p_n)) = - \frac{1}{2} ( -1 + \sum_{i=2}^{n} p_i ).
\end{align*}

By summing up, we obtain the same cumulative effect.
\\

(2) If $n=1$, then $P=P(p_1)$ is the trivial knot $O$, and hence $a_2(P)=0$. In the following, we assume $n \geq 3$.

First, we consider the case where $p_1 < 0$.
Let $k_+ = P(\underline{p_1}, p_2, ... , p_n)$ and $k_- = P(\underline{p_1 + 2}, p_2, ... , p_n)$.

By Proposition \ref{prop:lk}, we have $a_2(P(\underline{p_1}, p_2, ... , p_n))
- a_2(P(\underline{p_1 + 2}, p_2, ..., p_n)) = \frac{1}{2} \sum_{i=2}^{n} p_i$. 

By repeating the same computation
$-\frac{1}{2} p_1$ times, we have 
\begin{align*}
a_2(P(\underline{p_1}, p_2, ... , p_n))
& - a_2(P(\underline{p_1 + 2}, p_2, ..., p_n)) = \frac{1}{2} \sum_{i=2}^{n} p_i \quad, ... , \\
a_2(P(\underline{-2}, p_2, ... , p_n)) 
& - a_2(P(\underline{0}, p_2, ... , p_n)) = \frac{1}{2} \sum_{i=2}^{n} p_i.
\end{align*}

Here, $P(0, p_2, ... , p_n) \cong
T(2, p_2) \# ... \# T(2, p_n)$, 
hence $a_2(P(0, p_2, ... ,p_n)) = a_2(T(2, p_2)) + ... + a_2(T(2, p_n)) = \frac{1}{8} \sum_{i=2}^{n} (p_i^2 - 1)$.

By summing up, we obtain 
\begin{align*}
a_2(P(p_1, p_2, ... , p_n)) 
& = - \frac{1}{4} p_1 \sum_{i=2}^{n} p_{i} + \frac{1}{8} \sum_{i=2}^{n} (p_i^2 -1)\\
& = \frac{1}{8} \sum_{i=2}^{n} p_i^2 - \frac{1}{4} p_1 \sum_{i=2}^{n} p_{i} - \frac{1}{8} (n-1). 
\end{align*}

In the case where $p_1 > 0$, we set $k_+ = P(\underline{p_1 - 2}, p_2, ... , p_n)$ and $k_- = P(\underline{p_1}, p_2 , ... , p_n)$.

By Proposition \ref{prop:lk}, we have $a_2(P(\underline{p_1 - 2}, p_2, ... , p_n))
- a_2(P(\underline{p_1}, p_2, ..., p_n)) = \frac{1}{2} \sum_{i=2}^{n} p_i$.

By repeating the same computation
$\frac{1}{2} p_1$ times, we have 
\begin{align*}
a_2(P(\underline{p_1}, p_2, ... , p_n))
& - a_2(P(\underline{p_1 - 2}, p_2, ..., p_n)) = - \frac{1}{2} \sum_{i=2}^{n} p_i \quad, ... , \\
a_2(P(\underline{2}, p_2, ... , p_n)) 
& - a_2(P(\underline{0}, p_2, ... , p_n)) = - \frac{1}{2} \sum_{i=2}^{n} p_i.
\end{align*}

By summing up, we obtain the same cumulative effect.

The proof is complete.
\end{proof}

\subsection{Proofs of Theorems \ref{thm:odd} and \ref{thm:even}} \label{subsec33}

\odd*
\begin{proof}
By Proposition \ref{prop:pretzel},
we have $u^{\Delta}(P) = a_2(P)$.
The result then follows directly from Lemma~\ref{lemma:odd}.
\end{proof}

\even*
\begin{proof}
By Proposition \ref{prop:pretzel}, 
we have $u^{\Delta}(P) = a_2(P)$.
The result then follows directly from Lemma~\ref{lemma:even}.
\end{proof}

\section{Proof of Theorem \ref{thm:oddone}}\label{sec4}

\begin{remark}
It is known that there exist pretzel knots for which $u^{\Delta}(P) \ne |a_2(P)|$.
For example:
\[
\begin{array}{c|c|c|c|c}
P & \text{pretzel knot} & a_2(P) & u^{\Delta}(P)\\
\hline
8_2     & P(2,5,1)           &  0 & 2\\
8_5     & P(2,3,3)           & -1 & 3\\
8_{21}  & P(2,-3,-3,1)       &  0 & 2\\
9_8     & P(2,1,1,1,1,-3,1)  &  0 & 2\\
10_{46} & P(2,5,3)           &  0 & 4\\
10_{76} & P(2,1,1,3,3)       & -2 & 2 \text{ or } 4
\end{array}
\]
All of these pretzel knots are pretzel knots of even type.
\end{remark}

It is observed that, for pretzel knots of odd type, $u^{\Delta}(P) = |a_2(P)|$ in the examples considered. In fact, this equality holds for all pretzel knots of odd type with at most $9$ crossings.
Here, we investigate the following theorem (Theorem \ref{thm:oddone}).

\oddone*
\begin{proof}
By Claim \ref{claim:technique},
we have $d_G^{\Delta}(P(-1, p_2, p_3, ... , p_n), P(-1, p_2-2, p_3, ... , p_n)) \leq \frac{1}{2} ( \sum_{i=3}^{n}p_i - 1 )$. 

By repeating the same computation $\frac{1}{2} (p_2 - 1)$ times, we have
\begin{gather*}
d_G^{\Delta}(P(-1, \underline{p_2}, p_3, ... , p_n), P(-1, \underline{p_2-2}, p_3, ... , p_n)) \leq 
\frac{1}{2} ( \sum_{i=3}^{n}p_i - 1 ) 
\quad, ... , \\
d_G^{\Delta}(P(-1, \underline{3}, p_3, ... , p_n), P(-1, \underline{1}, p_3, ... , p_n)) \leq
\frac{1}{2} ( \sum_{i=3}^{n}p_i - 1 ). 
\end{gather*}

Here, $P(-1, 1, p_3, ... , p_n) \cong P(p_3, ... , p_n)$.
Let $P'=P(p_3, ... , p_n)$, then $P'$ is a positive pretzel knot of odd type.
By Theorem \ref{thm:odd}, we have $u^{\Delta}(P')=a_2(P')$.

Then, we obtain 
\begin{align*}
& d_G^{\Delta}(P(-1, p_2, ... , p_n), O) \\ & 
\leq d_G^{\Delta}(P(-1, p_2, p_3, ... , p_n), P(-1, 1, p_3, ... , p_n)) 
+ d_G^{\Delta}(P(-1, 1, p_3, ..., p_n), O) 
\\ &
\leq \frac{1}{4} 
( \sum_{i=3}^{n}p_i - 1 )(p_2 - 1)
+ \frac{1}{4} \sum_{3 \leq i<j }^{n} p_{i}p_{j} +  \frac{1}{8} \{ (n-2) - 1) \}
\\
& = - \frac{1}{4} \sum_{i = 2}^{n} p_{i} + 
\frac{1}{4} \sum_{2 \leq i<j }^{n} p_{i}p_{j} +  \frac{1}{8} (n-1)
\\
& = a_2(P) = |a_2(P)|.
\end{align*}
Therefore, we have $u^{\Delta}(P) \leq |a_2(P)|$.

By Proposition \ref{prop:a2}, we also have $u^{\Delta}(P) \geq |a_2(P)|$. 
Thus, we conclude that $u^{\Delta}(P)=|a_2(P)|$.
The proof is complete.
\end{proof}

\section{$\Delta$-unknotting number and signature}\label{sec5}

It is known that
$u^{\Delta}(K) \geq |\sigma(K)|/2$, where $\sigma(K)$ denotes the signature of a knot $K$ (see \cite{MaN}). 
In this section, we give examples demonstrating the usefulness of the inequality $u^{\Delta}(K) \geq |\sigma(K)|/2$.

\subsection{Case of $P(m,1,1,-3,1)$}
\label{subsec51}

\leavevmode

First, we give an example realizing the equality case of the inequality
$u^{\Delta}(K) \geq |\sigma(K)|/2$.

\begin{remark}
The following pretzel knots satisfy $u^{\Delta}(K)=1$ with at most 9 crossings.
\[
\begin{array}{c|c|c|c}
K & \text{pretzel knot}& |\sigma(K)|/2 & a_2(K) \\
\hline
3_1   & P(1,1,1)           & 1 &  1 \\
4_1   & P(1,1,2)           & 0 & -1 \\
6_2   & P(1,2,3)           & 1 & -1 \\
6_3   & P(2,1,-3,1)        & 0 &  1 \\
7_6   & P(2,1,1,-3,1)      & 1 &  1 \\
7_7   & P(1,1,1,-3,-3)     & 0 & -1 \\
8_{11}& P(3,1,1,-3,1)      & 1 & -1 \\
8_{13}& P(1,1,1,-3,-4)     & 0 &  1 \\
9_{12}& P(4,1,1,-3,1)      & 1 &  1 \\
9_{14}& P(1,1,1,-3,-5)     & 0 & -1 \\
9_{24}& P(3,1,-3,1,2)      & 0 &  1 \\
9_{28}& P(1,1,-3,2,1,-3)   & 1 &  1
\end{array}
\]
\end{remark}

We therefore obtain the following example.
\begin{example}
Let $K = P(m,1,1,-3,1)$ be a nontrivial pretzel knot, where $m$ is a nonzero integer. 
Then, we have $u^{\Delta}(K) = 1$ $\big(=|a_2(K)| \big)$, and
\[
|\sigma(K)|/2 =
\begin{cases}
1 & (m > 0), \\
0 & (m < 0).
\end{cases}
\]
For $m>0$, the knot $P(m,1,1,-3,1)$ satisfies 
$u^{\Delta}(K) = |\sigma(K)|/2$.
\end{example}

\begin{figure}[htbp]
    \centering
    \includegraphics[width=1\linewidth]{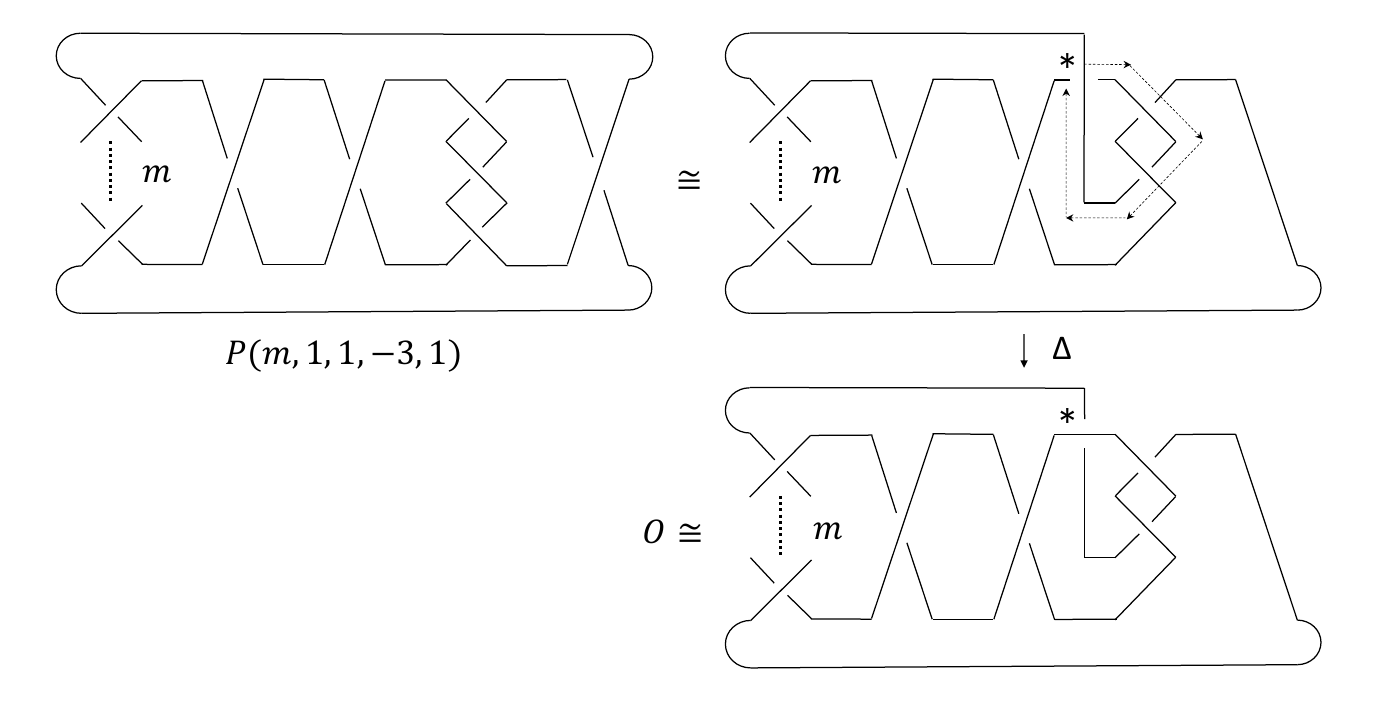}
    \caption{}
    \label{fig:placeholder}
\end{figure}

\subsection{Case of $P(2,3,m)$}
\label{subsec52}

\leavevmode

We give an example demonstrating the usefulness of the inequality $u^{\Delta}(K) \geq |\sigma(K)|/2$.

\begin{example}
Let $K = P(2,3,m)$ be a nontrivial pretzel knot, where $m$ is a positive odd integer. 
Then, we have
$|\sigma(K)|/2 = \frac{1}{2} (m+1)$ and 
$a_2(K) = \frac{1}{8} (m^2-4m-5)$.
\[
\begin{array}{c|c|c|c|c}
K & \text{pretzel knot}& |\sigma(K)|/2 & a_2(K) & u^{\Delta}(K)\\
\hline
6_2      & P(2,3,1)  & 1 & -1 & 1\\
8_5      & P(2,3,3)  & 2 & -1 & 3\\
10_{46}  & P(2,3,5)  & 3 &  0 & 4
\end{array}
\]

Among knots with at most $9$ crossings, $P(2,3,3)$ is the only knot for which 
$u^{\Delta}(K)$ cannot be determined using only $a_2(K)$ ( Proposition \ref{prop:a2} ), but can be determined
by additionally using the inequality
$u^{\Delta}(K) \geq |\sigma(K)|/2$.
\end{example}

\begin{figure}[htbp]
    \centering
    \includegraphics[width=0.8\linewidth]{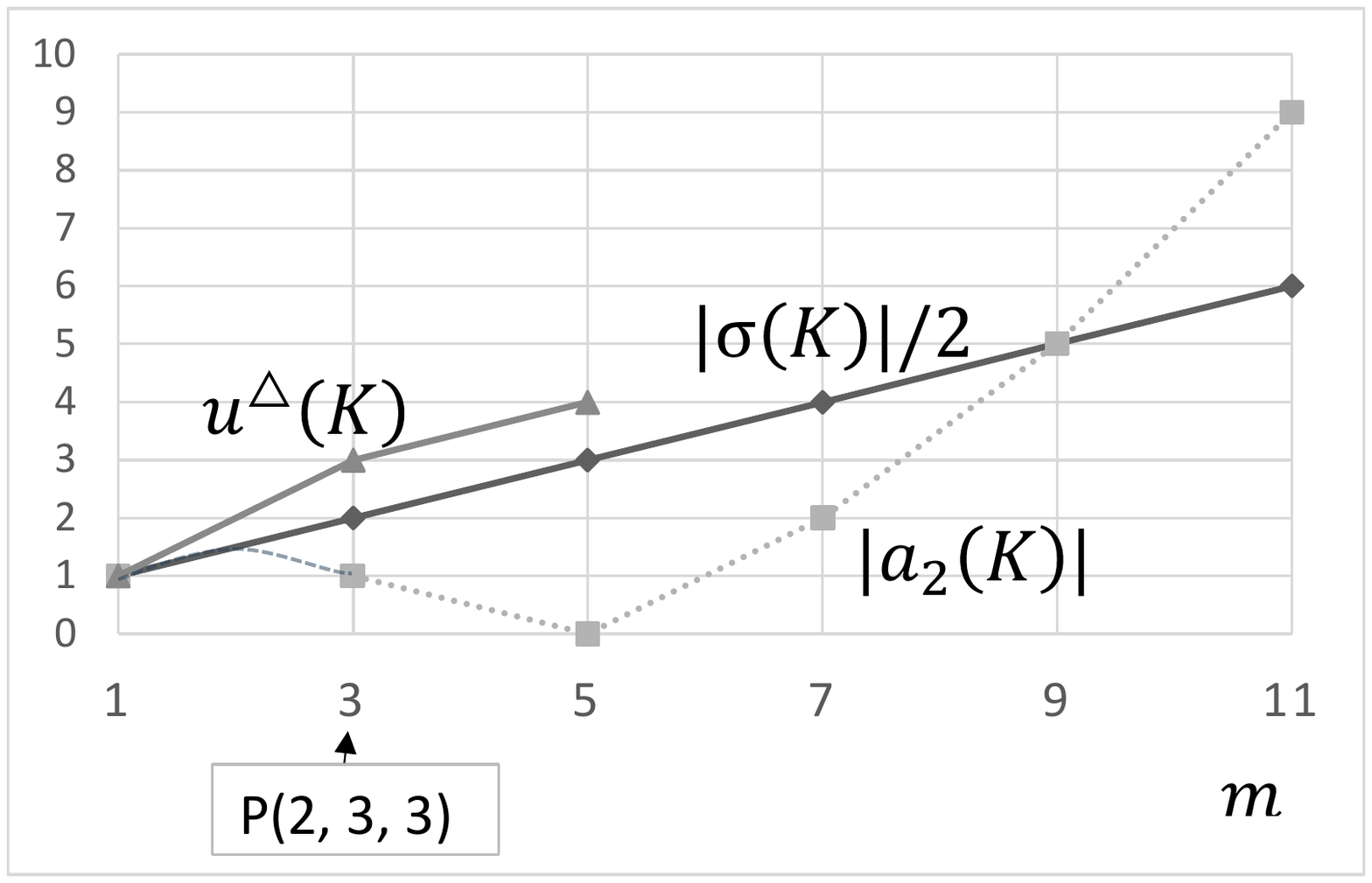}
    \caption{}
    \label{fig:placeholder}
\end{figure}

\section{$\Delta$-unknotting number for positive knots}\label{sec6}

In this section, we present several results and conjectures concerning the $\Delta$-unknotting number for positive knots.

\subsection{$\Delta$-Unknotting Number and Crossing Number}
\label{subsec63}

\leavevmode

The following Theorem is a consequence of Theorem \ref{thm:odd} and
Proposition \ref{prop:torus}.

\deltamax*
\begin{proof}
(1) (i) Let $P=P(p_1, p_2, ... , p_m)$ be a positive pretzel knot of odd type, where $p_i$ is a positive odd integer for $1 \leq i \leq m$ and $m$ is odd.

Positive pretzel knot of odd type is alternating.
Thus, $n = \sum_{i=1}^{m} p_{i}$. 
Then, we have
\[
\sum_{1 \leq i<j }^{m} p_{i}p_{j}
= \frac{1}{2} \{ (\sum_{i=1}^m p_i)^2 - \sum_{i=1}^m p_i^2 \}
= \frac{1}{2}(n^2 - \sum_{i=1}^m p_i^2).
\]

Now we estimate $\sum_{i=1}^m p_i^2$. Since each $p_i$ is a positive odd integer, we have $p_i \ge 1$, hence $p_i^2 \ge 1$.
Therefore, $\sum_{i=1}^m p_i^2 \ge m$.

Substituting this into the previous expression gives
\[
\sum_{1 \leq i<j }^{m} p_{i}p_{j}
\leq \frac{1}{2}(n^2 - m).
\]

By Theorem \ref{thm:odd}, we have
\[
u^{\Delta}(P) = 
\frac{1}{4} \sum_{1 \leq i<j }^{m} p_{i}p_{j} + \frac{1}{8} (m-1)
\leq
\frac{1}{8}(n^2 - m)+ \frac{1}{8} (m-1)
= \frac{1}{8} (n^2-1).
\]

Equality holds if and only if
$p_i^2 = 1$ for all $i$,
that is, $p_i = 1$ for all $i$.

\vspace{1em}
(1)(ii) Let $P=P(p_1, p_2, ... , p_m)$ be a positive pretzel knot of even type $A$, where $p_1$ is a positive even integer, and $p_i$ is a positive odd integer for $2 \leq i \leq m$ and $m$ is even. 

Positive pretzel knot of even type $A$ is alternating.
Thus, $n = \sum_{i=1}^{m} p_{i}$.
Then, we have
\[
\sum_{i=2}^m p_i^2
= (\sum_{i=2}^m p_i)^2- 2\sum_{2 \leq i<j }^{m} p_{i}p_{j}
\leq (\sum_{i=2}^m p_i)^2
=(n-p_1)^2.
\]

By Theorem \ref{thm:even}(1), we have
\begin{align*}
u^{\Delta}(P)
&=
\frac{1}{8} \sum_{i=1}^{m} p_{i}^2
+ \frac{1}{4} p_1 \sum_{i=2}^{m} p_i
- \frac{1}{8} (m-1) \\
&=
\frac{1}{8} ( \sum_{i=2}^{m} p_{i}^2 + p_1^2 )
+ \frac{1}{4} p_1 (n-p_1)
- \frac{1}{8} (m-1) \\
&\leq
\frac{1}{8} \left\{ (n-p_1)^2 + p_1^2 \right\}
+ \frac{1}{4} p_1 (n-p_1)
- \frac{1}{8} (m-1) \\
&=
\frac{1}{8} (n^2 - m + 1) \\
&\leq
\frac{1}{8} (n^2 - 1).
\end{align*}

Equality holds if and only if $m = 2$ and $p_2 = n-p_1$.

\vspace{1em}
(1)(iii) Let $P=P(p_1, p_2, ... , p_m)$ be a positive pretzel knot of even type $B$, where $p_1$ is a negative even integer, and $p_i$ is a positive odd integer for $2 \leq i \leq m$ and $m$ is odd. 

Assume that $n = |p_1| + \sum_{i=2}^{m} p_{i}$. Thus, $n$ is even. 
Then, we have
\[
\sum_{i=2}^m p_i^2
= (\sum_{i=2}^m p_i)^2- 2\sum_{2 \leq i<j }^{m} p_{i}p_{j}
=(n + p_1)^2 - 2\sum_{2 \leq i<j }^{m} p_{i}p_{j}.
\]

By Theorem \ref{thm:even}(2), we have
\begin{align*}
u^{\Delta}(P)
&=
\frac{1}{8} \sum_{i=2}^{m} p_{i}^2 - \frac{1}{4} p_1 \sum_{i=2}^{m} p_i - \frac{1}{8} (m-1) \\
&=
\frac{1}{8} \{ (n + p_1)^2 - 2\sum_{2 \leq i<j }^{m} p_{i}p_{j} \}
- \frac{1}{4} p_1 (n + p_1)
- \frac{1}{8} (m-1) \\
&=
\frac{1}{8} (n^2 - p_1^2 -m + 1- 2\sum_{2 \leq i<j }^{m} p_{i}p_{j}) \\
&\leq
\frac{1}{8} (n^2 - 6n +24).
\end{align*}

Equality holds if and only if $m = 3$,
$p_1 = -2$, and either $p_2 = 3$ or $p_3 = 3$.

\vspace{1em}
(2) It is known that the crossing number of the torus knot $T(2,n)$ is $n$. 
Let $2 \leq p < q$. Then the crossing number of $T(p,q)$ is $(p-1)q$.

We set $n = (p-1)q$. 
Then we compare
\[
\frac{1}{8}(n^2 - 1) = \frac{1}{8}\{(p-1)^2 q^2 - 1\}
\]
and
\[
\frac{1}{24}(p^2 - 1)(q^2 - 1).
\]

A direct computation shows that
\[
3 \{ (p-1)^2 q^2 - 1 \} - (p^2 - 1)(q^2 - 1) \ge 0
\]
for all integers $2 \le p < q$, hence
\[
\frac{1}{8}(n^2-1) \ge \frac{1}{24}(p^2-1)(q^2-1).
\]

Therefore,
\[
u^{\Delta}(T(2,n)) \geq u^{\Delta}(T(p,q)).
\]

Equality holds if and only if $p = 2$ and $q=n$.
\end{proof}

Here, $P(\underbrace{1, 1, ... , 1}_{\text{$n$}}) \cong P(p_1, n-p_1)
\cong T(2,n)$.

\deltamin*
\begin{proof}
Let $P=P(p_1, p_2, ... , p_m)$ be a positive pretzel knot of odd type, where $p_i$ is a positive odd integer for $1 \leq i \leq m$ and $m$ is odd.

Positive pretzel knot of odd type is alternating.
Thus, $n = \sum_{i=1}^{m} p_{i}$. 
Then, we have
\[
\sum_{1 \leq i<j }^{m} p_{i}p_{j}
= \frac{1}{2} \{ (\sum_{i=1}^m p_i)^2 - \sum_{i=1}^m p_i^2 \}
= \frac{1}{2}(n^2 - \sum_{i=1}^m p_i^2).
\]

By Theorem \ref{thm:odd}, we have
\begin{align*}
u^{\Delta}(P) & = 
\frac{1}{4} \sum_{1 \leq i<j }^{m} p_{i}p_{j} + \frac{1}{8} (m-1) \\
& =\frac{1}{8}(n^2 - \sum_{i=1}^m p_i^2)+ \frac{1}{8} (m-1) \\
& \geq \frac{1}{8}\{n^2 - \big((n-m+1)^2 + (m-1) \big) \}
+ \frac{1}{8} (m-1) \\
& = \frac{1}{8}\{n^2 - (n-m+1)^2\} \\
& \geq \frac{1}{2} (n-1)
\end{align*}

Equality holds if and only if
$m=3$, and up to permutation
$(p_1, p_2, p_3) = (1, 1, n-2)$.
\end{proof}

Here, $P(1,1,n-2) \cong C(n-2,2).$

\begin{example}
The following gives pretzel knots of odd type together with the values of $u^{\Delta}(K)$, arranged according to the crossing number $c(K)$.
\[
\begin{array}{c|c|c}
c(K) & \text{pretzel knot} & u^{\Delta}(K) \\
\hline
5 & 5_1 \cong P(1,1,1,1,1) & 3 \\
  & 5_2 \cong P(1,1,3)  & 2 \\
\hline
7 & 7_1 \cong P(1,1,1,1,1,1,1) & 6 \\
  & 7_2 \cong P(5,1,1)  & 3 \\
  & 7_4 \cong P(3,1,3)  & 4 \\
\hline
9 & 9_1 \cong P(1,1,1,1,1,1,1,1,1) & 10 \\
  & 9_2 \cong P(1,1,7)  & 4 \\
  & 9_5 \cong P(1,3,5)  & 6 \\
  & 9_{10} \cong P(-3,-1,-1,-1,-3)  & 8 \\
  & 9_{35} \cong P(3,3,3)  & 7 \\
\end{array}
\]
\end{example}

\begin{example}
The following gives pretzel knots of even type $A$ together with the values of $u^{\Delta}(K)$, arranged according to the crossing number $c(K)$.
\[
\begin{array}{c|c|c}
c(K) & \text{pretzel knot} & u^{\Delta}(K) \\
\hline
7 & 7_1 \cong P(2,5),P(4,3),P(6,1) & 6 \\
  & 7_3 \cong P(1,1,1,4)           & 5 \\
  & 7_5 \cong P(2,1,1,3)           & 4 \\
\hline
9 & 9_1 \cong P(2,7),P(4,5),P(6,3),P(8,1) & 10 \\
  & 9_7 \cong P(-3,-1,-1,-1,-1,-2)  & 5 \\
  & 9_9 \cong P(-4,-1,-1,-3)        & 8 \\
  & 9_{16} \cong P(2,3,1,3)         & 6 \\
\end{array}
\]
\end{example}

\begin{example}
The following gives pretzel knots of even type $B$ together with the values of $u^{\Delta}(K)$, arranged according to the crossing number $c(K)$.
\[
\begin{array}{c|c|c}
c(K) & \text{pretzel knot} & u^{\Delta}(K) \\
\hline
12 & 12_{n242} \cong P(-2, 3 ,7) & 12 \\
   & P(-2, 5, 5) & 11 \\
\end{array}
\]
\end{example}

\begin{example}
The following gives torus knots together with the values of $u^{\Delta}(K)$, arranged according to the crossing number $c(K)$.
\[
\begin{array}{c|c|c}
c(K) & \text{torus knot} & u^{\Delta}(K) \\
\hline
15 & T(2,15) & 28 \\
   & T(4,5)  & 15 \\
\hline
21 & T(2,21) & 55 \\
   & T(4,7)  & 30 \\
\hline
35 & T(2,35) & 153 \\
   & T(6,7)  & 70 \\
\end{array}
\]
\end{example}

The following conjectures are motivated by Theorem \ref{thm:deltamax}.

\conjcrtn*

\conjcrposi*

Since there are no positive knots with crossing numbers 4 or 6, these cases are excluded.
For crossing number 10, $u^{\Delta}(10_{139}) > u^{\Delta}(10_{124})=8$. Here, $10_{124} \cong P(-2,3,5) \cong T(3,5)$.

\begin{remark}
The following table summarizes, for each crossing number, the maximal value of the $\Delta$-unknotting number among prime knots and the knots realizing it.
For prime knots with at most 10 crossings, Conjectures \ref{conj:conjcrtn} and \ref{conj:conjcrposi} hold.
Furthermore, for completeness, if the crossing number is 11,
$|a_2(K)|$ is maximized by $T(2,11)$, and the maximum value is 15.
\[
\begin{array}{c|c|c|c|c|c|c}
c(K) & K & u^{\Delta}(K) & \text{positivity} & \text{alternation} & \text{torus knot} & \text{pretzel knot} \\
\hline
3  & 3_1      & 1  & positive & alternating & T(2,3) & P(1,1,1) \\
4  & 4_1      & 1  & --       & alternating & -- & P(1,1,2) \\
5  & 5_1      & 3  & positive & alternating & T(2,5) & P(1,1,1,1,1) \\
6  & 6_1      & 2  & --       & alternating & -- & -- \\
7  & 7_1      & 6  & positive & alternating & T(2,7) & P(1,1,1,1,1,1,1) \\
8  & 8_{19}   & 5  & positive & --          & T(3,4) & P(3,3,-2) \\
9  & 9_1      & 10 & positive & alternating & T(2,9) & P(1,1,1,1,1,1,1,1,1) \\
10 & 10_{139} & 9  & positive & --          & -- & -- \\
\end{array}
\]
\end{remark}

\vspace{1em}
\subsection{Positive Knots with $\Delta$-Unknotting Number One}
\label{subsec61}

\leavevmode

The following theorem is a consequence of Theorems~\ref{thm:odd}, ~\ref{thm:even}, 
and Proposition~\ref{prop:torus}.

\deltaone*
\begin{proof}
(1) By Theorems~\ref{thm:odd} and~\ref{thm:even}, the positive pretzel knots with $\Delta$-unknotting number one are $P(1,1,1)$, $P(2,1)$, or $P(-2,1,1)$.

(2) By Proposition~\ref{prop:torus}, the torus 
knots with $\Delta$-unknotting number one are $T(p,q)$ with $(|p|, |q|)=(2, 3)$ or $(3, 2)$. 

The proof is complete.
\end{proof}

The following conjecture is motivated by Theorem \ref{thm:deltaone}.

\conjpkdone*

\begin{remark}
For prime knots with at most 10 crossings, Conjecture \ref{conj:conjpkdone} holds.
\end{remark}

Since there is no almost positive knot of
$4$-genus (or unknotting number) one (see \cite{Ta}) and the $\Delta$-unknotting number is at least the $4$-genus (see \cite{MaN}), we obtain the following corollary.

\apkone*

\vspace{1em}
\subsection{Positive Knots with $u^{\Delta}(K) = |a_2(K)|$}
\label{subsec62}

\leavevmode

It is known that if $K$ is a torus knot, 
a positive pretzel knot, or a positive 3-braids, then $u^{\Delta}(K) = |a_2(K)|$.
This leads us to propose the following conjectures (Conjectures \ref{conj:conjpd} and \ref{conj:conjdd}).

\conjpd*

\conjdd*

Conjecture \ref{conj:conjpd} concerns the positivity of knots, whereas 
Conjecture \ref{conj:conjdd} concerns the positivity of minimal crossing diagrams.

\begin{remark}
For prime knots with at most 10 crossings, Conjectures \ref{conj:conjpd} and \ref{conj:conjdd} hold.
\end{remark}

For Conjecture \ref{conj:conjpd}, we verified that all positive knots with at most 10 crossings satisfy $u^{\Delta}(K) = |a_2(K)|$.

Similarly, for Conjecture \ref{conj:conjdd}, we verified that, for all knots $K$ with at most 10 crossings, if $u^{\Delta}(K) \neq |a_2(K)|$, then no minimal crossing diagram of $K$ is positive. In this verification, we examined the diagrams in \cite{Kaw}.

\section*{Acknowledgments}
The author would like to thank Professor Makoto Sakuma for his valuable advice and continuous support.

The author would also like to thank Professors Yoshiaki Uchida and Kouki Taniyama for pointing out relevant literature.

Finally, the author would like to thank his family for their support and understanding throughout this work.







\begin{thebibliography}{99}

\bibitem{DaM}
R. Díaz and P. M. G. Manchón, \textit{Pretzel knots up to nine crossings}, Topology and its Applications. \textbf{339} (2023), Article 108583. 


\bibitem{Kau} 
L.H. Kauffman, \textit{On knots}, Ann. of Math. Studies, vol.115,Princeton Univ. Press, 1990. 

\bibitem{Kaw}
A. Kawauchi, \textit{A Survey of Knot Theory},
Birkhäuser Verlag, Basel, 1996.

\bibitem{MaN} 
H. Murakami and Y. Nakanishi, \textit{On a certain move generating link-homology}, Math. Ann. \textbf{284} (1989),  75--89. 

\bibitem{NaNaU} 
K. Nakamura, Y. Nakanishi and Y. Uchida, \textit{Delta-unknotting number for knots}, J. Knot Theory Ramifications \textbf{7} (1998), 639--650. 



\bibitem{Naka3} 
K. Nakamura, \textit{Delta-unknotting number for two-bridge knots}, arXiv:2512.22970 [math.GT], 2025.

\bibitem{NaYa} 
Y. Nakanishi and M. Yamada, \textit{On Turk's head knots}, Kobe J. Math. \textbf{17} (2000), 119--130. 

\bibitem{Oka1} 
M. Okada, \textit{Delta-unknotting operations and the second coefficient of the Conway polynomial}, J. Math. Soc. Japan \textbf{42} (1990),  713--717. 



\bibitem{Ta}
K. Tagami, \textit{The Rasmussen invariant, four-genus and three-genus of an almost positive knot are equal}, Canadian Mathematical Bulletin \textbf{57} (2014), 431--438.

\bibitem{Tani} 
K. Taniyama, \textit{Pairs of knot invariants}, J. Knot Theory Ramifications \textbf{33} (2024), 2450030. 

\bibitem{Tani2} 
K. Taniyama, \textit{Unknotting numbers of diagrams of a given nontrivial knot are unbounded}, J. Knot Theory Ramifications \textbf{18} (2009), 1049--1063. 
\end{thebibliography}


\end{document}